\documentclass[preprint,12pt]{elsarticle}

\usepackage{amsmath,amsfonts,amsthm,amssymb,amscd}
\usepackage{hyperref}
\usepackage{tikz-cd}
\usepackage{enumerate}
\usepackage{braket}

\newcommand{\trm}[1]{\textrm{#1}}
\newcommand{\coh}{\textrm{Coh}}
\newcommand{\gm}{\G_m}

\newcommand{\sheafm}{\til{M_\lam}}
\newcommand{\iso}{\overset{\sim}{\longrightarrow}}

\newcommand{\disc}{\textrm{disc}}
\newcommand{\cont}{\textrm{cont}}
\newcommand{\cusp}{\textrm{cusp}}

\DeclareRobustCommand\longtwoheadrightarrow
     {\relbar\joinrel\twoheadrightarrow}

\newcommand{\resprod}{\prod^{\prime}}

\newcommand{\lam}{\lambda}

\newcommand{\ade}{\A}
\newcommand{\afinite}{\ade_{\textrm{fin}}}

\newcommand{\zhat}{\widehat{\Z}}

\def\A{\mathbb{A}}

\def\C{\mathbb{C}}

\def\G{\mathbb{G}}

\def\Q{\mathbb{Q}}
\def\R{\mathbb{R}}

\def\Z{\mathbb{Z}}

\def\R{\mathbb{R}}

\def\g{\mathfrak{g}}

\def\k{\mathfrak{k}}

\newcommand{\CalC}{\mathcal{C}}

\newcommand{\CalH}{\mathcal{H}}

\newcommand{\CalO}{\mathcal{O}}

\newcommand{\CalV}{\mathcal{V}}

\theoremstyle{definition}
\newtheorem{thm}{Theorem}[section]
\newtheorem{cor}[thm]{Corollary}
\newtheorem{prop}[thm]{Proposition}
\newtheorem{lem}[thm]{Lemma}
\newtheorem{defn}[thm]{Definition}
\newtheorem{remark}[thm]{Remark}

\newtheorem{notation}[thm]{Notation}

\newcommand{\Gam}{\Gamma}
\newcommand{\gam}{\gamma}

\newcommand{\rank}{\operatorname{rank}}
\newcommand{\Hom}{\operatorname{Hom}}

\newcommand{\End}{\operatorname{End}}

\newcommand{\img}{\operatorname{Img}}

\newcommand{\sym}{\operatorname{Sym}}

\newcommand{\bs}{\backslash}

\newcommand{\eps}{\epsilon}
\newcommand{\res}{\textrm{res}}

\newcommand{\abs}[1]{| #1 |}
\newcommand{\gl}{\operatorname{\mathfrak{gl}}}
\newcommand{\ip}[1]{\langle #1 \rangle}
\newcommand{\til}[1]{\widetilde{#1}}

\newcommand{\dual}{\vee}
\newcommand{\bsc}{\bar{S}_{K_f}}
\newcommand{\lss}{S_{K_f}}
\newcommand{\bul}{\bullet}
\newcommand{\hbul}{H^\bullet}
\newcommand{\gkmod}{(\g,K_\infty)}
\newcommand{\smooth}{\mathcal{C}^\infty}
\newcommand{\hecke}{\CalH}

\newcommand{\ul}[1]{\underline{#1}}

\newcommand{\Res}{\textrm{Res}}

\begin{document}
\begin{frontmatter}
\title{Inner cohomology of $GL_n$}
\author{Krishna Kishore}
\ead{venkatakrishna@iiserpune.ac.in}

\address{Venkata Krishna Kishore Gangavarapu,\\
Department of Mathematics,\\
Indian Institute of Science Education and Research,\\
Pune, 411008, India.}
\begin{abstract}
We give an explicit description of the inner cohomology of an adelic locally symmetric space of a given level structure attached to the general linear group of \textit{prime} rank $n$, with coefficients in a locally constant sheaf of complex vector spaces.  We show that for all primes $n$ the inner cohomology vanishes in all degrees for nonconstant sheaves, otherwise the quotient module of the inner cohomology classes that are not cuspidal  is trivial in all degrees for primes $n = 2,3$, and for all primes $n \geq 5$ it is trivial in all but finitely many degrees where it has a `simple'  description in terms of algebraic Hecke characters.
\end{abstract}


\fntext[]{This work is supported by the DST-SERB-NPDF fellowship SERB-PDF/2016/004056 of the Government of India.}

\end{frontmatter}

\newcommand{\hinnermodcusp}{H^\bullet_{!/\cusp}}
\section{Introduction}

Let $G$ be a connected reductive algebraic group over $\Q$ and $\Gam \subset G(\Q)$ an arithmetic subgroup. For an irreducible finite-dimensional complex representation $M$ of $G(\R)$, the group cohomology $H^*(\Gam, M)$ provides a concrete realization of some automorphic forms that are of number-theoretic interest. For example, if $G = SL_2$, and $\Gam$ is a congruence subgroup of level $N$, the well-known Eichler-Shimura isomorphism exhibits $H^1(\Gam, \C)$ as the span of modular forms of $\Gam$ of weight $2$. But,  these cohomology groups $H^\bullet(\Gam, M)$ `captures' only {\it some} automorphic forms and in fact, almost all automorphic forms do not appear in them; nevertheless, those that do appear have number-theoretic significance, partly justifying their study. \\

\noindent Henceforth, throughout this section, let $G = Gl_n$, and  $K_\infty = O(n) Z(\R)$ be the maximal compact modulo center subgroup of $G(\R)$, and let $K_f \subset G(\afinite)$ be an open compact subgroup that is neat [Definition \ref{neat_subgroup_defn}] (see \S \ref{adelic_setup} for notation). Let $(\rho_\lam, M_\lam)$ be the highest weight $G(\R)$-module associated to the dominant integral weight $\lam$ such that its central character $\omega_\lam$ is a type of an algebraic Hecke character (see \S \ref{algebraic_Hecke_character_defn} for definition). \\

\noindent The group cohomology $H^\bul(\Gam, M_\lam)$ is known to be isomorphic to the sheaf cohomology $H^\bul(\lss, \sheafm)$ of the adelic locally symmetric space $S_{K_f} := G(\Q) \bs G(\ade)/K_\infty K_f$ with coefficients in the locally system $\sheafm$ \textit{derived} from $M_\lam$, allowing one to understand the former in terms of the latter. The space $S_{K_f}$ is not compact in general, and has a compactification called the Borel-Serre compactification $\bsc$ (see \ref{Borel_Serre_compactification}), equipped with an  inclusion $\iota : \lss \hookrightarrow \bsc$ that is a homotopy equivalence and the canonical restriction $r : \bsc \to \partial \bsc$ onto the boundary $\partial \bsc = \bsc \setminus \lss$: 
$$
\lss \xrightarrow{i} \bsc \xrightarrow{r} \partial \lss $$
The coefficients on these spaces are given by the short exact sequence of sheaves
$$
0 \to i_{!} \sheafm \xrightarrow{i} i_* \sheafm \xrightarrow{r} i_* \sheafm/ i_{!} \sheafm \to 0
$$
where $i_{!} \sheafm$ is the sheaf extended by zero from $\lss$ to $\bsc$, and the quotient $i_* \sheafm/ i_{!} \sheafm$ is the sheaf $i_* \sheafm$ restricted to the boundary extended by zero to $\bsc$. Accordingly, there is a fundamental long exact sequence 
\begin{equation*}
\ldots \to H^{k-1}( \partial \bsc, \sheafm)\to H^k_c(\lss, \sheafm) \xrightarrow{i^k} H^k(\lss, \sheafm) \xrightarrow{r^k} H^{k+1}(\partial \bsc, \sheafm) \to \ldots
\end{equation*}
where the \textit{cohomology with compact supports} $
H^\bullet_c(\lss, \sheafm) := H^\bullet(\bsc, i_{!}\sheafm)
$. The  main goal of this article is to give an an explicit description, in the case where $G = GL_n$ with $n$ prime, of the \textit{inner cohomology}
$$
H^\bul_{!}(\lss, \sheafm) := \img( H^\bullet_c(\lss, \sheafm) \xrightarrow{i^\bullet} H^\bullet(\bsc, \sheafm)).
$$

\noindent In this paper, we use the `approximation' given by Borel and Garland \cite{BG}, namely the homomorphism of Hecke modules, which \textit{surjects} onto the subspace of $H^\bullet(\lss, \til{M})$ called the square-integrable cohomology $H^\bullet_{(2)}(\lss, \til{M})$ (see Definition \ref{square_integrable_cohomology}):
\begin{equation*}\label{intro_BG_map}
H^\bullet(\g, K_\infty, L^2(G(\Q)\bs G(\ade)/K_f, \omega_\lam^{-1}) \otimes M_\lam) 
\overset{\phi^\bullet_{BG}}{\longtwoheadrightarrow}
 H^\bullet_{(2)}(\lss, \sheafm)
\end{equation*}
where the coefficient system $L^2(G(\Q)\bs G(\ade)/K_f, \omega_\lam^{-1}) \otimes M_\lam$ of the Lie algebra cohomology is well-understood (see \S \ref{Langlands_spectral_decomposition}), thanks to the spectral decomposition of Langlands, that has a refinement due to M{\oe}glin and Waldspurger in the case of our interest, namely for $G = GL_n$. Together with the strong multiplicity-one result of Jacquet and Shalika  one obtains a satisfactory description of the domain of $\phi^\bullet_{BG}$ thereby of its image, which is $\hbul_{(2)}(\lss, \sheafm)$ containing the square-integrable cohomology $H^\bullet_{!}(\lss, \sheafm)$ (see below \ref{intro_filtration_cohomology}). \\

\noindent Let $\CalH_f :=  \CalC_c(G(\afinite) // K_f, \C)$ be the  Hecke algebra of $K_f$-bi-invariant compactly supported complex-valued functions $\phi : G(\afinite) \to \C$ with the algebra structure given by convolution. The space 
$$
L^2(\omega_\lam^{-1}) := L^2(G(\Q)\bs G(\ade)/K_f, \omega_\lam^{-1}),
$$
is a $G(\R) \times \CalH_f$-module, and is the direct sum of the discrete spectrum $L^2_{\disc}$ which is the maximal closed subspace spanned by irreducible $G(\R) \times \hecke_f$-modules, and its orthogonal complement called the {\it continuous spectrum} $L^2_{\cont}(\omega_\lam^{-1})$. The discrete spectrum contains the {\it cuspidal spectrum} $L^2_{\cusp}(\omega_\lam^{-1})$, and there is a natural inclusion whose image is called \textit{cupsidal cohomology}  
\begin{equation*}\label{intro_BG_map}
H^\bullet_{\cusp}(\lss, \bsc) := \img(H^\bullet(\g, K_\infty, L^2_{\cusp}(\omega_\lam^{-1}) \otimes M_\lam) \hookrightarrow H^\bullet(\lss, \sheafm)).
\end{equation*}
The full cohomology $H^\bullet(\lss, \sheafm)$ has the following filtration as $\CalH_f$-modules:
\begin{equation}\label{intro_filtration_cohomology}
\hbul_\cusp(\lss, \sheafm) \subset \hbul_{!}(\lss, \sheafm) \subset \hbul_{(2)}(\lss, \sheafm) \subset \hbul(\lss, \sheafm).
\end{equation}
Since cuspidal cohomology is well-understood, namely by the inclusion above spanned by the cuspidal automorphic forms $L^2_{\cusp}(\omega_\lam^{-1})$, it is natural to study the quotient $\CalH_f$-module 
$$
H^\bullet_{!/ \cusp}(\lss, \sheafm) := \hbul_{!}(\lss, \sheafm) \Big/ \hbul_{\cusp}(\lss, \sheafm).
$$ 
In other words, we give an explicit description of the inner cohomology classes $H^{\bullet}_{!}(\lss, \sheafm)$ that are not cuspidal in the case where $G = GL_n$, with $n$ a prime number; in the particular case of primes $n =2,3$ the description is \textit{even} simpler. The main results of the article are as follows: \\

\noindent Let $\coh_{\infty}(G, \lam)$ be  the set of isomorphic classes of essentially-unitary irreducible representations  $\CalV_{\pi_\infty}$ of $G(\R)$ with nontrivial $\gkmod$-cohomology with coefficients in $M_\lam$, and let $\coh_{(2)}(G, K_f, \lam)$ be the set of isomorphism classes of absolutely-irreducible $\CalH_f$-modules $\pi_f$ for which there exists a $\pi_\infty \in \coh_\infty(G,\lam)$ such that $\Hom_{G(\R) \times \CalH_f}(V_{\pi_\infty} \otimes V_{\pi_f}, V_{(2)}(\omega_\lam^{-1})) \neq 0$, and let  
\begin{equation*}\label{analysis_2_residual_part_of_Borel_Garland_map}
\Res_{f}(\lam) :=  \bigoplus_{ \substack{\pi_f \in \coh_{(2)}(G, K_f, \lam) \\
\textrm{type}(\pi_f) = \omega_\lam^{1/n}} }
\pi_f.
\end{equation*}

\begin{thm}\label{intro_main_thm}
Assume that $n$ is a prime number. For all primes $n \geq 2$, the quotient module $H^\bullet_{!/\cusp}(\lss, \sheafm)$ vanishes if $\sheafm$ is not isomorphic to the constant sheaf $\C$. So, suppose otherwise, i.e. $\sheafm \cong \C$, and let 
$S^0 = \Set{ 2l-1 | 1 < l \leq n, \; l \trm{ odd }}$, then  

\begin{enumerate}
\item 
for prime $n =2,3$, the module $H^\bullet_{!/\cusp}(\lss, \C) = 0$, and

\item 
for all primes $n \geq 5$,
\begin{equation}\label{possible_cases}
H^k_{!/\cusp}( \lss, \C) \cong \begin{cases}
0  & \trm{ for } k \not \in S^0.  \\
\ker( r^k|_{\Phi^k_{BG}(\Res_f(\lam)}) & \trm{ for } k \in S^0.
\end{cases}
\end{equation}
\end{enumerate}

\end{thm}

\noindent The paper is organized as follows. In \S \ref{basic_setup} we recall the notion of 
adelic locally symmetric space $\lss$ and the structure of sheaf $\sheafm$ on it defined by $M_\lam$, and the notion algebraic Hecke characters. In \S \ref{cohomology_of_arithmetic_groups} we recall the cohomology of arithmetic groups, and define required definitions that are directly relevant to our article. We discuss the Hecke module structure of cohomology groups, and the associated  fundamental isomorphism with the $\gkmod$-cohomology or relative Lie-algebra cohomology. 
In \S \ref{decomposing_cohomology} we recall the coarse decomposition of the space $L^2(G(\Q)\bs G(\ade)/K_f, \omega_\lam^{-1})$ due to Langlands into various subspaces, and a finer one also due to Langlands refined further by M{\oe}glin and Waldspurger in the our case of intersect, namely the $GL_n$ case. Finally, in \S \ref{main_result} we determine the contribution of residual spectrum to the inner cohomology, and prove Theorem \ref{intro_main_thm}.

\section*{Acknowledgements}
\noindent It is a great pleasure to thank A.Raghuram for suggesting the question, and Dipendra Prasad for helpful discussions.

\section{Notation}

The notation $G$ always denotes the the general linear group $GL_n$ defined over $\Q$.  Consider the inclusions $G \supset P \supset B = T U \supset T \supset Z$ of subgroups all defined over $\Q$, where $P$ is a parabolic subgroup, $B$ the standard Borel subgroup of upper triangular matrices, $T$ the maximal torus of  diagonal matrices, $U$ the unipotent subgroup of strict upper triangular matrices, and $Z$ the center of $G$.  We call a parabolic subgroup of $G$, such as $P$, {\it standard} if it contains $B$. For a $\Q$-algebra $A$, let $G(A)$ denote the group of $A$-valued points of $G$, and $G_A$ the extension of scalars of $G$ from $\Q$ to $A$.

The dimension of a subgroup $K$ of $G$ is denoted by $\dim K$, and its $\Q$-rank by $\rank K$. The notation $G^\circ$ denotes the connected component of the identity of $G$, and $\pi_0(G(\R))$ is the group of connected components of $G(\R)$.
The notation $N_G(K)$ denotes the normalizer of $K$ in $G$. The Lie algebra of $G$ is denoted by $\g$, and its universal enveloping algebra by
$\mathfrak{U}(\g)$.

\section{Basic setup}\label{basic_setup}

\subsection{Adelic setup}\label{adelic_setup}

\noindent Let $\ade = \R \times (\resprod_p\Q_p) = \ade_\infty \times \afinite$ be the ring of adeles over $\Q$, where $\ade_\infty = \R$ is the {\it archimedean} component, and $\afinite = \resprod_p \Q_p$ is the {\it nonarchimedean} component, which is the restricted direct product of the local fields $\Q_p$ as $p$ runs through the set of finite primes. Then $G(\ade) := G(\R) \times G(\afinite) := G_\infty \times G_f$.  Fix a subgroup $K_\infty =  O(n) Z(\R) = O(n) Z(\R)^\circ$, the maximal compact modulo center subgroup. The {\it symmetric space} associated to the pair $(G_\infty, K_\infty)$ is the quotient space $X_{\sym}:= G_\infty/K_\infty$. 

Let $\Gam \subset G(\Q)= GL_n(\Q)$ be an arithmetic subgroup, i.e. for all congruence subgroups $\Gam'$ the intersection $\Gam \cap \Gam'$ is of finite index both in $\Gam$ and $\Gam'$. Suppose $\Gam$ has no torsion, then its natural action on $X_{\sym}$ by left multiplication is properly discontinuous and free, resulting in a locally symmetric space $\Gam \bs X_{\sym}$. \\

\noindent Let $\rho : G \to GL(M)$ be a finite-dimensional complex rational representation of $G$. It defines a local system $\ul{M}$ of complex vector spaces on $\Gam \bs X_{\sym}$, and one has 
$$
H^\bullet(\Gam \bs X_{\sym}, \ul{M}) \cong H^\bullet(\Gam, M)	.
$$
The cohomology on the left hand side is computed with the aid of the de Rham complex (and that on the right is the ordinary group cohomology). Passing further on to the adelic setup so as to bring in the results of automorphic representations, let $K_f \subset G(\afinite)$ be a compact open subgroup, and consider the following construction, wherein the action of $G(\Q)$ is by left multiplication and all the maps are the canonical projections (see \cite[Chapter 3]{Ha1}):
\[
\begin{tikzcd}\label{role_change_diagram}
X_{\sym} \times G(\afinite) \ar{r}{\pi'} \ar{d}{\Pi'}  &X_{\sym}  \times G(\afinite)/K_f \ar{d}{\pi} \\
G(\Q) \bs \Big( X_{\sym} \times G(\afinite) \Big) \ar{r}{\Pi} & G(\Q)\bs \Big( X_{\sym} \times G(\afinite)/K_f \Big)
\end{tikzcd}
\]

\vspace{0.5cm}
\noindent Let $
S_{K_f} := G(\Q)\bs \Big( X_{\sym} \times G(\afinite)/K_f \Big) = G(\Q) \bs G(\ade)/K_\infty K_f$,  called the {\it adelic locally symmetric space}. It can be equipped with the coefficient sheaf $\til{M}$, obtained from the representation $(\rho,M)$, whose sections on an open set $V \subset S_{K_f}$ are the set $\til{M}(V)$ of locally constant functions $s: \pi^{-1}(V) \to M$ satisfying
$$
s(\gam (x_\infty K_\infty, g_f K_f)) = \rho(\gam) s((x_\infty K_\infty, g_f K_f)), \; \;
\trm{ for all } \gam \in \Gam, u \in \pi^{-1}(V).
$$
\begin{remark}Eventually, we view sheaf cohomology groups $\hbul(\lss, \sheafm)$ as Hecke modules; see \S \ref{Hecke_action}. In particular these groups should be equipped with the $\Gam$-action on the left or equivalently the $K_f$-action on the right. Informally speaking, this is analogous to the strong approximation theorem that aids in trading transformation property under the left-action of $SL(2,\Z)$ of the modular forms on $SL(2,\R)$ with the transformation property under the right-action of the maximal compact subgroup $SO(2,\R)$ of automorphic forms on $SL(2,\R)$.
\end{remark}
Consider the natural inclusion $\til{M}  \hookrightarrow \til{M} \otimes \afinite$, and given a section $s \in \til{M}(V)$ associate a map $s_1 : \pi'^{-1} ({\pi^{-1}(V)})  \to \til{M} \otimes \afinite$ defined by
$$
s_1(x_\infty, g_f) := g_f^{-1} s(x_\infty K_\infty , g_f K_f)
$$
where $g_f$ acts on the second factor of $M \otimes \afinite$. Evidently $s_1(\gam (x_\infty, g_f)) = s_1(x_\infty, g_f)$ for all $\gam \in G(\Q)$, so that  $s_1$ factors through the map
$$
s_2 : G(\Q) \bs \Big( G(\R)/K_\infty \times G(\afinite) \Big) \to M \otimes \afinite,
$$
and defines a sheaf $\til{M} \otimes \afinite$ on the space $\lss$. Alternatively, since $\Pi^{-1}(V) = \Pi'(\pi'^{-1} \circ \pi^{-1}(V))$, we obtain a sheaf $\til{M \otimes \afinite}$ whose sections on an open set $V \subset S_{K_f}$ are the set $\til{M \otimes \afinite}(V)$ of locally constant functions $s: \Pi^{-1}(V) \to M\otimes \afinite$ satisfying $
s(x_\infty K_\infty, g_f k_f) = k_f^{-1} s(x_\infty, g_f)$, for all  $x_\infty \in G_\infty$,  $g_f \in G_f$,  $k_f \in K_\infty$. In summary, the sheaf $\til{M} \otimes \afinite$ defined in terms of the $G(\Q)$-action on $M$ on the left is identified with the sheaf $\til{M \otimes \afinite}$ defined in terms of the natural right action of $K_f$ on $M \otimes \afinite$ on the right.

\subsection{Topological structure of $S_{K_f}$:}\label{topological_structure_of_lss}

The quotient $
X_{\sym} \times G(\afinite)/ K_f = G(\Q) \bs G(\afinite)/K_f$ under the natural action of $G(\Q)$ on $G(\afinite)/K_f$ is a finite set $\Set{g_f^{1}, g_f^{2}, \ldots, g_f^{l}}$, and a connected component is of the form 
$$
X_i := G(\A)^{\circ} (\eps, g_f^{(i)}) K_f/ K_\infty K_f
$$
where $\eps \in \pi_0(G(\R))$.  Let $\Gam_i \subset G(\Q)$ be its stabilizer, which is an arithmetic subgroup of $G(\Q)$. Then we have  $S_{K_f} = \coprod_{i=1}^l	\Gam_i \bs X_i$ (see \cite[\S 1.1]{Ha2})

\begin{defn}\label{neat_subgroup_defn}
The subgroup $K_f$ is said to be neat if the $\Gam_i$ are torsion free.
\end{defn}
\begin{remark}
The sheaf cohomology groups $H^\bullet(S_{K_f}, \til{M})$ are known to be isomorphic to finite direct sum of the cohomology groups of the form $H^\bullet(\Gam \bs G(\R)/K_\infty, \til{M})$ for an appropriate arithmetic subgroup $\Gam \subset G(\Q)$, and under mild restrictions both on $S_{K_f}$ and $\til{M}$. 
If the stabilizers $\Gam_i$ has no torsion, then they act freely, so that the connected components are locally symmetric. This is true in our case of interest, i.e. $G = GL_n/\Q$. Indeed, the stabilizer $\Delta$ of a point $g = (g_\infty, g_f^{i}) K_\infty K_f$ in $\Gam_i$ is a congruence subgroup in the connected component of the unit group $\Set{1, -1}$ of the center $Z(\Q) = \Q^\times$, hence trivial. But, if we consider groups over an arbitrary number field $F \neq \Q$ then,  we have to pass onto the action of $\Gam \bs \Delta$ above to get a locally symmetric space, since the unit group $\CalO_F^\times$ is nontrivial as a consequence of  Dirichlet's unit theorem; accordingly we have to consider the group cohomology $\hbul(\Gam_i, M) := \hbul( ( \Gam_i/\Delta_i) \bs X_i, \til{M}	)$. 
\end{remark}

\subsection{Sheaf structure on $S_{K_f}$}\label{sheaf_str}
\noindent The group of rational characters $X^*(T) := \Hom(T, \gm)$ of the maximal torus $T$ is a free abelian group group of rank $n$. It is equipped with the standard basis $e_i: \textrm{diag}(t_1, \ldots, t_n) \to t_i$. The structure of $X^*(T)$ is more transparent if we pass onto $X^*(T) \otimes_\Z \Q$ and consider the fundamental basis associated to the standard basis. The {\it fundamental weights} $\gam_i \in X^*(T)_\Q$ are characterized by the conditions that they act on the center $Z$ by $z \mapsto z^i$, and they satisfy the following relations: for all $1 \leq i \leq n-1$, and $1 \leq j \leq n$,
$$
2 \ip{\gam_i, e_j - e_{j+1}}/\ip{e_j-e_{j+1}, e_{j}- e_{j+1}} = \delta_{ij}.
$$ 
In particular, the determinant character $\delta := e_1 + \ldots + e_n$ spans $X^*(Z)_\Q$, and the set $\Set{\gam_1, \ldots, \gam_{n-1}, \delta}$ is a basis of $X^*(T)_\Q$ called the {\it fundamental basis of $T$.} Let then $\gam  = \sum_{i=1}^{n-1} a_i \gam_i + d \delta $; it is said to be {\it integral} if $a_i \in \Z$, $n d \in \Z$ and $nd \equiv \sum_{i=1}^{n-1} i(a_i -1) \pmod{n}$. An integral weight is said to be {\it dominant} if, in addition, the coefficients $a_i\geq 0$. \\

\noindent Suppose that the representation $M$ is absolutely irreducible.  By the highest-weight theory, up to isomorphism, $M_\C$ is isomorphic to  $M_\lam \otimes \C$ where $M_\lam$ is the highest weight module associated to the dominant integral weight $\lam \in X^*(T)_\Q$. Throughout this article, we consider the restriction
$$
\rho_\lam := \rho(\C)|_{G(\R)} : G(\R) \to GL(M_\lam \otimes \C).
$$
Henceforth, we work \textit{only} with the module $M_\lam \otimes \C$ exclusively, so, for ease of notation, we drop the second factor $\C$ in $M_\lam \otimes \C$ and simply write as $M_\lam$.  

Now, with this abuse of notation, consider the associated sheaf $\til{M}_\lam$ on the adelic locally symmetric space $S_{K_f}$ [\ref{adelic_setup}]. Let $\omega_\lam: Z(\R) \to \C^\times$ be the central character of $\rho_\lam$, and it is given by $\omega_\lam(z) = z^{nd}$ with $d$ the coefficient of the determinant character $\delta$ in the expression of the character $\lam : T(\R) \to \C^\times$	 in the fundamental basis. Then $\omega_\lam(-I_n) = -1 \; \textrm{or} \; 1$. Suppose it is $-1$ and consider the stalk $\til{(M_\lam)}_x$ for some $x \in S_{K_f}$: 
$$
\til{(M_\lam)}_x  = \Set{ s_x : \pi^{-1}(x)	\to M_\lam | s_x(\gam \cdot u)=  \rho_\lam(\gam) s_x(u), \;  \; \gam \in G(\Q),  u \in \pi^{-1}(x)}
$$
Since the representative section $s$ of $s_x$ is locally constant the germ $s_x$ is constant, hence 
$$
s(u) = s(-I_n \cdot u) = \omega_\lam(-I_n) s(u) = - s(u), \; \; u \in \pi^{-1}(x),
$$
forcing  the stalk $(\sheafm)_x$ to be trivial, whence the sheaf $\sheafm$ is also trivial; note that here we used the `thickened' aspect of $K_\infty$ namely that $K_\infty$ is $O(n) \R^*$ rather than just $O(n)$, which is also considered in the literature. Therefore, to have $\til{M_\lam}$ to be not zero identically, we must  restrict our attention to representations $\rho_\lam$ whose central characters $\omega_\lam : Z(\R) \to \C^\times$ satisfy $\omega_\lam(-I_n) = 1$. This implies in particular that $\omega_\lam$ is determined by its values on the connected component of the identity $Z(\R)^\circ \cong \R^\times_{>0}$. In other words, the central $\omega_\lam$ is a type of algebraic Hecke character; see below, and also \cite[\S 2.5]{Ha2}.

\subsection{Algebraic Hecke characters}\label{algebraic_Hecke_characters}
\begin{defn}\label{algebraic_Hecke_character_defn}
An \textit{algebraic Hecke character of a torus $S$ of type $\gam \in X^*(S)_\Q$ defined over $\Q$} is a continuous group homomorphism $\phi : S(\Q) \bs S(\ade) \to \C^\times$ such that
$\phi|_{S(\R)^\circ} = \gam_\infty^{-1}|_{S(\R)^\circ}$, where $\gam_\infty : S(\R) \hookrightarrow S(\C) \to \C^\times$.
\end{defn}

\noindent Applying the definition to our situation, keeping in view of the assumption that $\omega_\lam(- I_n) = 1$,  we have that $\omega_\lam: Z(\R) \to \C^\times$ is the type of algebraic Hecke character $\phi : Z(\Q) \bs Z(\ade) \to \C^\times$ such that $\phi|_{Z(\R)} = \omega_{\lam}^{-1}$. The center $Z(\ade) \cong \ade^\times \cong \Q^\times \times \R^\times_{>0} \times \zhat$, and therefore $\phi$  is determined by its `finite part' $\phi_f : \zhat \to \C^\times$ (its infinite part is given by $\omega_\lam$), which has finite order because $\zhat$ is compact, and therefore must factor through the map $(Z/N\Z)^\times \to \C^\times$ for some positive integer $N$; the least such $N$ is called the {\it conductor} of the character $\phi$. Consequently, the algebraic Hecke characters (in our situation) of type $\omega_\lam$ are parametrized by primitive Dirichlet characters.

\subsection{Summary}\label{summary_basic_setup}
\noindent Let us summarize the assumptions about the principal objects of our study: $G:= GL_n$, $K_\infty = O(n) \R^*$, $
S_{K_f} = G(\Q) \bs G(\ade)/K_\infty K_f$ is the adelic locally symmetric space attached to the pair $(G_\infty, K_\infty)$ and for some choice of compact open subgroup $K_f \subset G(\afinite)$ that is neat. We study the sheaf cohomology groups $H^\bullet(S_{K_f}, \til{M_\lam})$ where $(\rho_\lam, M_\lam)$ is the highest weight $G(\R)$-module associated to the dominant integral weight $\lam$ such that its central character $\omega_\lam$ is a type of an algebraic Hecke character.

\section{Cohomology of arithmetic groups}
\label{cohomology_of_arithmetic_groups}
\noindent In this section we recall several notions related to the sheaf cohomology. The reader may refer to \cite[Chapter 2]{Ha} for the formal properties of sheaf cohomology, and \cite[Chapter 2]{Ha1} for their interpretation as Hecke modules in our context.
\subsection{Hecke action}\label{Hecke_action}
\noindent The groups $\hbul(\lss, \sheafm)$
are functorial with respect to $K_f$. Indeed, passing onto a smaller compact open subgroup $K_f' \subset K_f$ (which is necessarily of finite index) yields a surjective map $
\pi_{K_f, K_f'}: S_{K_f'} \to \lss $
with finite fibers, and hence a map on cohomology 
$$
\pi_{K_f, K_f'}^\bullet : H^\bullet(\lss, \sheafm) \to H^\bullet(S_{K_f'}, \sheafm).
$$ The family  $ \{ \hbul(S_{K_f}, \til{M}), \pi_{K_f, K_f'}^\bullet \}$ indexed  by $K_f$ is a directed system with \textit{the} direct limit
$$
\hbul(S^G, \sheafm) = \varinjlim_{K_f} \hbul(\lss, \sheafm).
$$
The limit $\hbul(S^G, \sheafm)$ has a natural action of $\pi_0(G(\R)) \times G(\afinite)$ by right multiplication; for $(k_\infty, g_f) \in \pi_0(G(\R)) \times G(\afinite)$, the induced multiplication map $
m_{(k_\infty,x_f)} : S_{K_f} \iso S_{x_f^{-1} K_f x_f}$
is an isomorphism such that $(m_{(k_\infty, x_f)})_*(\sheafm) \cong \sheafm$, hence passing onto the limit results in the desired action. The cohomology with fixed level $K_f$ is obtained by taking the $K_f$-invariants under this action: $
\hbul(\lss, \sheafm) = \hbul(S^G, \sheafm)^{K_f}$. \\

\noindent Let $\CalH_f :=  \CalC_c(G(\afinite) // K_f, \C)$ be the \textit{Hecke algebra} of $K_f$-bi-invariant compactly supported functions $\phi : G(\afinite) \to \C$, with the algebra structure given by convolution:
$$
(h_1 \ast h_2)(g_f) = \mathop{\int}_{G(\afinite)} h_1(x_f) h_2(x_f^{-1}g_f) dx_f
$$
where the Haar measure $dx_f$ is normalized such that $K_f$ has unit volume. Clearly the characteristic function $\chi_{K_f}$ is the identity element of $\CalH_f$. The action of the group $G(\afinite)$ induces an action of $\CalH_f$ on the cohomology $\hbul(\lss, \sheafm)$ by
$$
T_h(v) = \mathop{\int}_{G(\afinite)} h(x_f) (x_f \cdot v) dx_f, \; \;  v \in \hbul(\lss, \sheafm).
$$
which is a finite sum: let $K_f' \subset K_f$ be the stabilizer of $v$, necessarily of finite index, then 
$$
T_h(v) = [K_f : K_f'] \sum_{a_f} \sum_{\xi_f \in G_f/K_f'} c_{a_f} \chi_{K_f a_f K_f}(\xi_f) (\xi_f \cdot v).
$$
is $K_f$-invariant. Therefore $T_h(v) \in \hbul(\lss, \sheafm)$, and since
\begin{align*}
T_{h_1 \ast h_2} &= \mathop{\int}_{G(\afinite)} (h_1 \ast h_2)(x_f) (x_f \cdot v) dx_f \\
&= \mathop{\int}_{G(\afinite)} \mathop{\int}_{G(\afinite)}h_1(y_f) h_2(y_f^{-1} x_f) dy_f (x_f \cdot v) dx_f \\
&= \mathop{\int}_{G(\afinite)} h_1(y_f ) y_f \cdot \Big(\mathop{\int}_{G(\afinite)} h_2(z_f) (z_f \cdot v) dz_f \Big) d(y_f z_f)\\
&= \mathop{\int}_{G(\afinite)} h_1(y_f)  \; \Big( y_f \cdot  T_{h_2}(v) \Big) dy_f = T_{h_1}(T_{h_2}(v))
\end{align*}
the map $\CalH_f \to \End_\C(\hbul(\lss, \sheafm))$ given by $h \mapsto T_h$ is a representation of the Hecke algebra $\CalH_f$ which is in fact finite-dimensional since $\hbul(\lss, \sheafm)$ is a finite-dimensional complex vector space  (see de Rham isomorphism \eqref{de_Rham_isomorphism}).

\newcommand{\fullcoh}{\hbul(\lss, \sheafm)}

\subsection{Borel-Serre compactification}\label{Borel_Serre_compactification}

\noindent We now turn to the topological aspects of $\lss$ (with $K_f$ neat) that will yield some more information about $\fullcoh$. In general, the space $\lss$ is not compact. In fact the associated adelic locally symmetric space of any general connected reductive group over $\Q$ is compact if and only if the the group is anisotropic over $\Q$, i.e. has no proper parabolic subgroups defined over $\Q$ \cite[Page 277]{Bo}. Certainly then in our case, namely $GL_n/\Q$, the space $\lss$ is not compact. Borel and Serre constructed a compactification $\bsc$ of $\lss$ by  `adding' the boundary $
\partial \bsc := \bigcup_{P} \partial_P \bsc$,
where $P$ runs through the (finitely many) {\it standard} representatives of $G(\Q) $-conjugacy classes of proper $\Q$-parabolic subgroups; the {\it Borel-Serre compactification} $\bsc = \lss \cup \partial \bsc,$ is a compact manifold with corners and $\dim \partial  \bsc  = \dim \bsc - 1$ (see \cite{BS}). The Borel-Serre compactification $\bsc$ is equipped with an inclusion $\iota : \lss \hookrightarrow \bsc$ that is a homotopy equivalence, and a canonical restriction $r : \bsc \to \partial \bsc$:
$$
\lss \xrightarrow{i} \bsc \xrightarrow{r} \partial \lss $$
The coefficients on these spaces are obtained by the canonical short exact sequence of sheaves
$$
0 \to i_{!} \sheafm \xrightarrow{i} i_* \sheafm \xrightarrow{r} i_* \sheafm/ i_{!} \sheafm \to 0
$$
where $i_{!} \sheafm$ is the sheaf extended by zero from $\lss$ to $\bsc$, and the quotient $i_* \sheafm/ i_{!} \sheafm$ is the sheaf $i_* \sheafm$ restricted to the boundary extended by zero to $\bsc$. 
\begin{defn}
The \textit{cohomology with compact supports} or compactly-supported cohomology is defined by 
$$
H^\bullet_c(\lss, \sheafm) := H^\bullet(\bsc, i_{!}(\sheafm)),
$$
and the image of $i^\bul$ is called the {\it inner or interior cohomology} and denoted $\hbul_{!}(\lss, \sheafm)$. The cohomology $\hbul(\partial \bsc, \sheafm)$ is called the {\it boundary cohomology}.
\end{defn}
\noindent The short exact sequence of sheaves yields the following fundamental long exact sequence equipped with $\CalH_f$-action \cite[Chapter 3]{Ha1},
\begin{equation}\label{Borel_Serre_fundamental_exact_sequence}
\ldots \to H^{k-1}( \partial \bsc, \sheafm)\to H^k_c(\lss, \sheafm) \xrightarrow{i^k} H^k(\lss, \sheafm) \xrightarrow{r^k} H^{k+1}(\partial \bsc, \sheafm) \to \ldots
\end{equation}

\vspace{0.5cm}
\begin{notation}\label{notation}
We make the following notation for ease of reference: $\hbul_{?}(\lss, \sheafm)$ where the symbol $?$ takes values in the set $\Set{ \trm{`empty'}, c, !, \partial}$. For example, by $H^\bullet_\partial(\lss, \sheafm)$ we mean $H^\bullet(\partial \bsc, \sheafm)$, and likewise for other symbols too. Let us note further that by the symbol ? = `empty' we mean $H^\bullet(\lss, \sheafm)$, i.e. the full or ordinary cohomology. 

When the coefficient system $\sheafm$ is clear from the context, we further simplify $\hbul_{?}(\lss, \sheafm)$ to $\hbul_{?}$.
\end{notation}
\begin{remark}\label{degree_0_compact_cohomology}
Note that the beginning of the fundamental exact sequence \eqref{Borel_Serre_fundamental_exact_sequence} is
$$
0 \to H_c^0 \to H^0 \to H^0_\partial \to \ldots;
$$ 
In particular, observe that the map $H_c^0 \to H^0$ is an injection, due to the fact the global-sections functor is left exact. 
\end{remark}

\subsection{Relative Lie algebra cohomology}\label{relative_Lie_algebra_cohomology}

\noindent Consider the \textit{de Rham complex} which is the resolution of the constant sheaf $\C$  by the sheaf of $M_\lam$-valued smooth forms $\Omega^\bullet(\lss, M_\lam)$ on $S_{K_f}$. The groups $H^\bullet(S_{K_f}, \til{M_\lam})$ are computed through de-Rham complex and with the aid of the \textit{de Rham isomorphism}
\begin{equation}\label{de_Rham_isomorphism}
\hbul(\lss, \sheafm) \cong \hbul(\Omega^\bullet(S_{K_f}, \til{M}_\lam)^\Gam).
\end{equation}
As an aside, let us note that this interpretation of sheaf cohomology in terms of deRham cohomology implies that the Poincar\'{e} duality holds on the sheaf cohomology groups as well, namely for all $0 \leq i \leq \dim X_{\sym}$ there exists a nondegenerate pairing, with $\sheafm^\dual$ the sheaf dual to $\sheafm$: 
\begin{equation}\label{Poincare_duality}
H^i(\lss, \sheafm) \times H^{d-i}_c(\lss, \sheafm^\dual) \to \C. 
\end{equation}

\vspace{0.5cm}
\noindent The de-Rham cohomology also has an interpretation in terms of the $\gkmod$-cohomology which we briefly recall now. The coefficient space $
V(\omega_\lam^{-1}) := \smooth(G(\Q)\bs G(\ade)/K_f, \omega_\lam^{-1})$, which is the space of smooth functions $\phi : G(\ade) \to \C$ satisfying
\begin{equation}\label{transformation_law}
\phi(g_0 z_\infty g_\infty K_f) = \omega^{-1}(z_\infty) \phi(g_\infty), \; \;  g_0 \in G(\Q), g_\infty \in G_\infty, k_f \in K_f,  z_\infty \in Z_\infty,
\end{equation}
is equipped with $G(\ade)$ action by right translation, which upon differentiation (in the $g_\infty$-variable) yields a $\g$-action. Hence we see that $V(\omega_\lam^{-1})$ is  a $\gkmod \times G(\afinite)$-module, hence also a $\gkmod \times \CalH_f$-module where the Hecke algebra acts by convolution. 

Let $V(\omega_\lam^{-1})^{(K_\infty)} \subset V(\omega_\lam^{-1})$ be the subspace of $K_\infty$-invariant vectors. It is a sub-$(\g,K_\infty)$-module. The {\it $(\g,K_\infty)$-cohomology or the relative Lie algebra cohomology} is defined as the cohomology of the complex
$
\Hom_{K_\infty}(\wedge^\bullet(\g/\k), V(\omega_\lam^{-1})^{(K_\infty)} \otimes M_\lam) 
$ where the action of $K_\infty$ on the exterior powers $\wedge^\bullet \g/\k$ is the one induced from the adjoint representation of $K_\infty$ in $\g/\k$ (see \cite[Chapter I]{BW}). The relation of the $\gkmod$-cohomology to the sheaf cohomology is based the following canonical isomorphism of complexes, which is compatible with the action of Hecke algebra:
\begin{equation}\label{de_Rham_Lie_algebra_isomorphism}
\Omega^\bullet_?(\lss, \sheafm) \cong \Hom_{K_\infty}(\wedge^\bullet(\g/\k), V_?(\omega_\lam^{-1})\otimes M_\lam) \; \; \; \;\; \; \textrm{ where } ? = \textrm{empty}, c
\end{equation}
where $V_c(\omega_\lam^{-1})$ consisting of compactly supported functions of $V(\omega_\lam^{-1})$. Isomorphisms \eqref{de_Rham_isomorphism} and \eqref{de_Rham_Lie_algebra_isomorphism} hints at the analysis of $V(\omega_\lam^{-1})$, therefore, in general of the Hilbert space obtained from its completion using a suitable norm:

\begin{defn}\label{square_integrable_function} 
The subspace of {\it square-integrable} functions 
$$
V_{(2)}(\omega_\lam^{-1}) := \smooth_2(G(\Q) \bs G(\ade)/K_f, \omega_\lam^{-1})  \subset V(\omega_\lam^{-1})
$$
is the subset of  $f  \in V(\omega_\lam^{-1})$ satisfying
\begin{equation}\label{square_integrable_defn}
\mathop{\int}_{G(\Q) Z(\R)^\circ \bs G(\ade)} \abs{(Uf)(g)}^2 \abs{\omega_\lam(g)}^2 dg < \infty.
\end{equation}
for all elements $U \in \mathfrak{U}(\g)$. Its completion, with respect to the norm defined by \eqref{square_integrable_defn} is denoted by $L^2(G(\Q)\bs G(\ade)/K_f,\omega_\lam^{-1})$.
\end{defn}

\begin{defn}\label{square_integrable_cohomology}
The \textit{square integrable cohomology} $\hbul_{(2)}(\lss, \sheafm)$ is sub-$\CalH_f$-module of $\hbul(\lss, \sheafm)$ consisting of those cohomology classes with a square-integrable function as a representative \textit{that is also a closed form} in 
$\Omega^\bullet(\lss, \sheafm$) (see \eqref{de_Rham_isomorphism}); for details about the motivation for this definition, see \cite[Chapter 3]{Ha1}.
\end{defn}

\noindent Finally, and clearly, we have the filtration as $\CalH_f$-modules, of the full cohomology :
\begin{equation}\label{basic_filtration}
\hbul_{!} \subset \hbul_{(2)} \subset \hbul; \; \; \trm{ see } \ref{notation}.
\end{equation}
\begin{remark}
Notice that $\hbul_{c}$ hence $\hbul_{!}$ is defined by geometric means, while $\hbul_{(2)}$ is defined by analytic means. In the next section we define \textit{cuspidal cohomology} by algebraic means.
\end{remark}

\section{Decomposing cohomology}
\label{decomposing_cohomology}

\subsection{Langlands spectral decomposition}\label{Langlands_spectral_decomposition}
\noindent Consider the Hilbert space  
$
L^2(\omega) := L^2(G(\Q) \bs G(\ade)/K_f, \omega)
$ (see \eqref{transformation_law} and Definition \eqref{square_integrable_function}). It is a $G(\R) \times \CalH_{f}$-module, where $G(\R)$ acts by unitary transformations and $\hecke_{f}$ by right convolution. Due to Langlands \cite{La}, the space $L^2(\omega)$ is the direct sum of the \textit{discrete spectrum} $L^2_{\disc}(\omega)$ and the {\it continuous spectrum} $L^2_{\cont}(\omega)$, where $L^2_{\disc}(\omega)$ is the maximal closed subspace spanned by irreducible $G(\R) \times \hecke_f$-modules, and $L^2_{\cont}(\omega)$ is the orthogonal complement of $L^2_{\disc}(\omega)$.

 A  representation occurring in $L^2_{\disc}(\omega)$ will be called {\it discrete}. The discrete spectrum contains the {\it cuspidal spectrum} $L^2_{\cusp}(\omega)$, namely the closed subspace spanned by functions $f \in L^2_{\disc}(\omega)$ such that the integral over $U(\Q) \bs U(\ade)$ of $f$, and all its right-translates under $G(\ade)$, vanishes, where  $U$ is the unipotent radical of any proper parabolic subgroup, and the measure is normalized so that $U(\Q) \bs U(\ade)$ has unit volume. The complement of $L^2_{\cusp}(\omega)$ in $L^2_{\disc}(\omega)$ is  called the {\it residual spectrum} $L^2_{\res}(\omega)$. The decomposition of discrete spectrum in to cuspidal spectrum and residual spectrum has a refinement in the $GL_n$ case due to Langlands \cite{La} and due to M{\oe}glin and Waldspurger \cite{MW}. The description of these results involves several notions, but we recall only those that are directly relevant to this article; for more details the reader may refer to the articles of Arthur \cite{Ar} \cite{Ar1}.

According to Borel and Casselman  \cite[\S 4]{BC}, the contribution of the continuous
spectrum to the $\gkmod$-cohomology is trivial. Therefore we may, and do, restrict our attention only to the discrete spectrum henceforth. 
\subsection{Residual spectrum}\label{residual_spectrum}
\noindent We use these notions in our special case summarized in \S \ref{summary_basic_setup}. First, consider the decomposition of $L^2_{\disc}(\omega_\lam^{-1})$ indexed by central characters $\omega : \Q^\times \bs \ade^\times \to \C^\times$ of type $\omega_\lam$:
$$
L^2_{\disc}(\omega_\lam^{-1}) =  \bigoplus_{\substack{\omega: \Q^\times \bs \ade^\times \to \C^\times  \\ \omega_\infty = \omega_\lam^{-1}}} L^2_{\disc}(\omega).
$$
We now analyse the structure of a summand $L^2_{\disc}(\omega)$. Consider the set of tuples $(L,W)$, where $L = GL(N_1) \times \ldots \times GL(N_m)$ is a \textit{standard} Levi subgroup of a (standard) parabolic subgroup, and $W = W_1 \otimes \ldots \otimes W_m$ is an irreducible subspace of the space of cuspidal automorphic representations of $L(\Q) \bs L(\ade)$. 

Two such tuples $(L,W)$ and $(L',W')$ are defined to be equivalent if there exists an $m-$tuple  $\underline{s} = (s_1, \ldots, s_m)$ of complex numbers such that the representation defined by $(L',W')$ is conjugate (by an element in $L(\ade)$) to the one defined by $(L,W[\underline{s}])$, where 
$$
W[\underline{s}] = W_1[s_1] \otimes \ldots \otimes W_m[s_m],  \;  \; \textrm{ where } W_i[s_i]= \Set{ \phi \abs{\det}^{s_i} | \; \phi \in W_i}.
$$
The {\it cuspidal support} of $\Psi \in \Xi$ is the set of all tuples in the equivalence class (see \cite[Lemma 6]{Ar}).

\vspace{0.5cm}
\noindent Let $\Xi$ be the set of equivalence classes, and  $\Xi^\circ$ be the subset of those equivalence classes $\Psi \in \Xi$ such that $\Psi$ contains an element $(L,W)$ satisfying
$$
N_1 = \ldots = N_m, \; \; \textrm{and} \; \; V_1 = \ldots = V_m.  
$$
The assumption that the central character of the  representation $(L,W)$ is equal to $\omega$ implies that there is precisely one such element $(L,W)$  in the equivalence class $\Psi$ (see \cite[\S 1]{MW}). Due to Langlands \cite{La}, we have the following decomposition:
\begin{equation}\label{spectral_decomposition}
L^2_{\disc}(\omega) = \bigoplus_{\Psi \in \Xi} L^2_{\disc}(\omega)_\Psi
\end{equation}
which has a refinement due to M{\oe}glin--Waldspurger \cite{MW} in the $GL_n$ case.
\begin{thm}\label{mw_thm}
(M{\oe}glin, Waldspurger)
Let $\Psi \in \Xi$. Then
\begin{enumerate}
\item
If $\Psi \not \in \Xi^0$, then $L^2_{\disc}(\omega)_\Psi \cap L^2_{\disc}(\omega) = 0$.
\item
If $\Psi \in \Xi^\circ$, then $L^2_{\disc}(\omega)_\Psi \cap L^2_{\disc}(\omega)$ is irreducible and isomorphic to an unique irreducible quotient of the induced representation $I(V,\underline{s})$, with $\underline{s} = ( \frac{m-1}{2}, \ldots, \frac{1-m}{2})$, obtained by normalized parabolic induction from $M$ to $G(\ade)$ of the representation $V_1[s_1]\otimes \ldots V_m[s_m]$.
\end{enumerate}
\end{thm}

\begin{cor}\label{mw_cor}
(M{\oe}glin, Waldspurger)
Let $W$ be a representation of $T/\Q$, so that the representation $(T,W)$ is of the form $\mu_1 \otimes \ldots \otimes \mu_n$ where $\mu_i : \Q^\times \bs \ade^\times \to \C^\times$ is a Hecke character over $\Q$. Let $\Psi \in \Xi$ be the equivalence class containing $(T,W)$. Then,
\begin{enumerate}
\item
only representations of the form $\mu \otimes \ldots \mu$  such that $\mu^n =\omega$, in the equivalence class $\Psi$, contribute to the residual spectrum $L^2_{\res}(\omega)$. 
\item  $L^2_{\disc}(\omega) \cap L^2(\omega)_{\Psi}$ is  isomorphic to the space spanned by the representation $\pi = \otimes^\prime_p \pi_p$, where $\pi_p = \mu_p \circ \det$ for all primes $p$.
\end{enumerate}
\end{cor}
\noindent In summary for $G = GL_n$ only the indexing set $\Xi^\circ$ is relevant in the direct sum \eqref{spectral_decomposition}:
\begin{equation}\label{MW_spectral_decomposition}
L^2_{\disc}(\omega) = \bigoplus_{\Psi \in \Xi^\circ} L^2_{\disc}(\omega)_\Psi \cong L^2_{\cusp}(\omega) \bigoplus \underbrace{\Big( \bigoplus_{\substack{\mu: \Q^\times \bs \ade^\times \to \C^\times \\ \mu^n = \omega}} \mathop{\otimes^\prime}_{p \leq \infty} (\mu_p \circ \det) \Big)}_{L^2_{\res}(\omega)}.
\end{equation}
Our main goal eventually is to understand the contribution of $L^2_{\res}(\omega)$ to the inner cohomology with the aid of the Borel-Garland map $\Phi^\bullet_{BG}$; see below \eqref{Borel_Garland_map}.

\subsection{Cuspidal cohomology}\label{filtration_cohomology}

Let $V_{\cusp}(\omega_\lam^{-1}) := \smooth_{\cusp}(G(\Q) \bs G(\ade)/K_f, \omega_\lam^{-1})$ be the subset of  the space of smooth cusp forms of  $V(\omega_\lam^{-1})$; smooth in the archimedean component, and locally constant in nonarchimedean components. The {\it cuspidal cohomology} is defined by
$$
\hbul_{\cusp}(\lss, \sheafm) := \hbul(\g, K_\infty, V_{\cusp}(\omega_\lam^{-1}) \otimes M_\lam).
$$
Now consider the filtration of $V(\omega_\lam^{-1})$:
\begin{equation}\label{space_filtration}
V_{\cusp}(\omega_\lam^{-1}) \subset V_{(2)}(\omega_\lam^{-1}) \subset V (\omega_\lam^{-1})
\end{equation}
Let $\coh_{\infty}(G, \lam)$ be the set of isomorphic classes of essentially-unitary (unitary up to twist by a central character) irreducible representations  $\CalV_{\pi_\infty}$ of $G(\R)$ with nontrivial $\gkmod$-cohomology with coefficients in $M_\lam$. By a result of Harish-Chandra, this set is finite. For $\CalV_\infty \in \coh_\infty(G, \lam)$, put $V_{\pi_\infty} = (\CalV_{\pi_\infty})^{K_\infty}$, which is a $\gkmod$-module, and consider the following spaces of homomorphisms (see \cite[\S 3.2.3]{HR}):
\begin{equation*}
\begin{split}
W_{\pi_\infty} &:= \Hom_{G(\R)}(V_{\pi_\infty}, V(\omega_\lam^{-1})) \\
W_{\pi_\infty \otimes \pi_f}^{(2)} &:= \Hom_{G(\R) \times \CalH_f}(V_{\pi_\infty} \otimes V_{\pi_f}, V_{(2)}(\omega_\lam^{-1})) \\
W_{\pi_\infty \otimes \pi_f}^{\cusp} &:= \Hom_{\gkmod \times \CalH_f}(V_{\pi_\infty} \otimes V_{\pi_f}, V_{\cusp}(\omega_\lam^{-1}))
\end{split}
\end{equation*}
Informally, the cardinality of these sets gives multiplicity (or \textit{w}eights, whence the notation $W$,) of the representation given by the respective domains in their respective target spaces. \\

\noindent Let $\coh_{(2)}(G, K_f, \lam)$, resp. $\coh_{\cusp}(G, K_f, \lam)$, be the set of isomorphism classes of absolutely-irreducible $\CalH_f$-modules $\pi_f$ for which there exists a $\pi_\infty \in \coh_\infty(G,\lam)$ such that $W_{\pi_\infty \times \pi_f}^{(2)} \neq 0$, resp. $W_{\pi_\infty \times \pi_f}^{\cusp} \neq 0$. From \ref{space_filtration}, it follows that 
$$
\coh_{\cusp}(G, K_f, \lam) \subset \coh_{(2)}(G, K_f, \lam) \subset \coh_\infty(G, \lam)
$$
Due to Jacquet--Shalika \cite{JS}, we have $\dim W_{\pi_\infty \otimes \pi_f}^{\cusp} \leq 1$. On the other hand, note that Theorem \ref{mw_thm} of M{\oe}glin-Waldspurger implies $\dim W_{\pi_\infty \otimes \pi_f}^{(2)} \leq 1$.\\

\noindent Finally, we have the following result of Borel and Garland \cite{BG} that `approximates' square-integrable cohomology $H^\bullet_{(2)}(\lss, \sheafm)$, namely that there is a \textit{surjective map} $\Phi_{BG}^\bullet$ of $\CalH_f$-algebras (\cite[\S 3.2.3]{HR}): 
\begin{equation}\label{Borel_Garland_map}
\begin{split}
\mathop{\bigoplus}_{\pi_\infty \in \coh_\infty(G, \lam)} \mathop{\bigoplus}_{\pi_f \in \coh_{(2)}(G, K_f, \lam)} &W_{\pi_\infty \otimes \pi_f}^{(2)} \otimes \hbul(\g,K_\infty, V_{\pi_\infty} \otimes M_\lam) \otimes V_{\pi_f} \xrightarrow{\Phi_{BG}^\bullet} \\
&\hbul_{(2)}(\lss, \sheafm).
\end{split}
\end{equation}
In particular the square-integrable cohomology, and hence its subspace the inner cohomology, are semisimple as $\CalH_f$-modules. On the other hand, due to Borel \cite{Bo1}, there is a canonical map of $\CalH_f$-modules, which is an isomorphism:
\begin{equation}\label{Borel_cusp_map}
\begin{split}
\bigoplus_{\pi_\infty \in \coh_\infty(G, \lam)} \bigoplus_{\pi_f \in \coh_{\cusp}(G, K_f, \lam)}&W_{\pi_\infty \otimes \pi_f}^{\cusp} \otimes \hbul(\g,K_\infty, V_{\pi_\infty} \otimes M_\lam) \otimes V_{\pi_f} \longrightarrow \\
&\hbul_{\cusp}(\lss, \sheafm)
\end{split}
\end{equation}

\noindent Therefore, comparing \eqref{Borel_cusp_map} and \eqref{Borel_Garland_map} we see that the cuspidal cohomology is contained in the inner cohomology \cite[\S 3.2.3]{HR}. Finally, taking in to account of the filtration \eqref{basic_filtration}, we obtain the following refinement of the filtration of the the full cohomology $H^\bullet$:
\begin{equation}\label{filtration}
\hbul_\cusp \subset \hbul_{!} \subset \hbul_{(2)} \subset \hbul.
\end{equation}
The main object of study in this article is the quotient $\CalH_f$-module 
$$
\hbul_{!/\cusp} := \hbul_{!}(\lss, \sheafm) \Big/ \hbul_{\cusp}(\lss, \sheafm).
$$

\section{Main result}\label{main_result}

\noindent We establish the main result of this article in this section. As always we work in the setup of \S \ref{summary_basic_setup}. In this section, we impose an additional hypohteis, \textit{which is crucial for the results of this article}, namely that \textit{$n \geq 2$ is a prime number}. To emphasize further, we suppose that the rank of $G = GL_n$, which is $n$, is a prime number.\\

\begin{prop}
Assume the hypothesis of \S \ref{summary_basic_setup}. Suppose that $n$ is a prime number. Then
\begin{equation}\label{residual_decomposition}
L^2_{\res}(\omega_\lam^{-1}) 
=  \Big( \omega_\lam^{-1/n} \circ \det \Big) \bigotimes \Big( \bigoplus_{\pi: \textrm{type}(\pi_f) = \omega_\lam^{1/n}} \pi_f \Big).
\end{equation}

\end{prop}
\begin{proof}
The proper standard $\Q$-parabolic subgroups of $G$ are in one-to-one correspondence with the nontrivial partitions of $n$, namely the partition $n = n_1 + \ldots + n_r$,  where the summands $n_i \geq 1$ corresponds to the the parabolic subgroup that has unipotent radical $U = GL_{n_1} \times \ldots \times GL_{n_r}$. Since $n$ is a prime number, there is a unique partition $n = n_1 + \ldots + n_r$ such that $n_1 = \ldots = n_r$, namely the one with all $n_i = 1$: it corresponds to the standard Borel subgroup $B$.

 Accordingly, the set $\Xi^0$ consists of equivalence classes with unique representatives $(B, W)$ where $W$ is a one-dimensional representation of the torus $T$ (see Theorem \ref{mw_thm}). Hence, given a central character $\omega$ of $G$ of type $\omega_\lam^{-1}$, the direct summand $L^2_{\res}(\omega)$ in $L^2_{\res}(\omega_\lam^{-1})$ is spanned by the one-dimensional automorphic representation $\pi(\omega) = \mu \circ \det$ where $\mu$ is a Hecke character, necessarily unitary, such that $\mu^n = \omega$; see \eqref{MW_spectral_decomposition}:

\begin{equation*}
L^2_{\res}(\omega_\lam^{-1}) = \bigoplus_{\substack{\omega :\Q^\times \bs \ade^\times \to \C^\times \\ \omega_\infty = \omega_\lam^{-1}}} \pi(\omega) 
=  \Big( \omega_\lam^{-1/n} \circ \det \Big) \bigotimes \Big( \bigoplus_{\pi: \textrm{type}(\pi_f) = \omega_\lam^{1/n}} \pi_f \Big).
\end{equation*}
\noindent The last equality follows from the fact that any algebraic Hecke character of a given type is uniquely determined uniquely by its finite component (see \eqref{algebraic_Hecke_character_defn}).
\end{proof}

The cuspidal cohomology is contained in the inner cohomology \eqref{filtration} and injects into the square-integrable cohomology \eqref{Borel_cusp_map}. On the other hand, the cuspidal cohomology is the image of the cuspidal spectrum which is disjoint from the residual spectrum, it follows that the inner cohomology classes are obtained from the residual spectrum and is contained in the image of the map \eqref{Borel_Garland_map}. With the aid of the residual decomposition \eqref{residual_decomposition} we deduce that the set of inner cohomology classes in $\hbul_{(2)} \supset \hbul_{!} \supset \hbul_{\cusp}$, {\it that are not cuspidal}, must be contained in the image of $
\hbul(\g,K_\infty, L^2_{\res}(\omega_\lam^{-1}) \otimes \sheafm)$ under the Borel-Garland map $\Phi^\bullet_{BG}$

\begin{remark}
Note that the map \eqref{Borel_Garland_map} is surjective onto $\hbul_{(2)}$, and it is not necessarily the case that $\hbul_{!} = \hbul_{(2)}$. Also, it is not necessarily the case that the image of  $\Res_f(\lam)$ generates the inner cohomology that are not cuspidal. All we know at the moment is that it generates square-integrable cohomology with the aid of Borel-Garland map $\Phi^{\bul}_{BG}$. \\
\end{remark}

\vspace{0.5cm}
\noindent Before we prove the main result we establish some elementary results. Let $\pi$ be an automorphic representation of $G(\ade)$; its archimedean component $\pi_\infty$ can be identified with a $\gkmod$-module on which the center $Z(\mathfrak{U}(\g))$ of $\mathfrak{U}(\g)$ acts by scalars, and the resulting map $Z(\mathfrak{U}(\g)) \to \C$ is called the infinitesimal character of $\g$.

\begin{lem}\label{Wigner_lemma}
The $\gkmod$-cohomology $\hbul(\g,K_\infty, (\omega_\lam^{-1/n} \circ \det ) \otimes M_\lam)$ is nontrivial, only if $M_\lam \cong \omega_\lam^{1/n} \circ \det$. In that case
$$
\hbul(\g,K_\infty, (\omega_\lam^{-1/n} \circ \det ) \otimes M_\lam) \cong \hbul(\g,K_\infty,\C).
$$
\end{lem}

\begin{proof}
Wigner's lemma \cite[Chapter I, Corollary 4.2]{BW} gives a necessary condition for the nonvanishing of the factor  $\hbul(\g,K_\infty, (\omega_\lam^{-1/n} \circ \det ) \otimes M_\lam)$, namely that the representations $(\omega_\lam^{-1/n} \circ \det)$ and $M_{\lam}^\dual$ have the same infinitesimal character, forcing the representation $\rho_\lam : G(\R) \to GL(M_\lam)$ to be $\omega_\lam^{1/n} \circ \det$; in other words, the coefficients of the cohomology $\hbul(\g,K_\infty, (\omega_\lam^{-1/n} \circ \det ) \otimes M_\lam)$ must be the trivial module $\C$.
\end{proof}

\begin{lem}\label{Haefliger_lemma}

Let $S^0 = \Set{ 2l -1 | 1 < l \leq n, \; l \trm{ odd }}$. Then
$$
\hbul(\g,K_\infty, \C) \cong H^\bullet(SU(n)/SO(n),\C) \cong \bigwedge^{*}[ \Set{\xi_i}_{i \in S^0}],
$$
the exterior algebra over $\C$ generated by symbols $\xi_i$ indexed by $S^0$.
\end{lem}

\begin{proof}
The $(\gl_n,O(n))$-cohomology $H^{\bullet}(\gl_n, O(n), \C)$ is isomorphic to the exterior algebra over $\C$ generated by  elements taken one in degrees $2k-1$, with $k$ odd, and $k \leq n$ (see \cite[Page 28]{Hae}. In our case, where instead of $O(n)$ we took $K_\infty = O(n) Z(\R)^\circ$, the exterior algebra $H^\bullet(\g,K_\infty, \C)$ has no generators in degree $1$, because 
\begin{align*}
H^\bullet(\g,O(n), \C) &= H^\bullet(U(n)/O(n),\C) 
= H^\bullet(U(1) \times SU(n)/SO(n),\C)  \\
&= H^\bullet(U(1), \C) \otimes H^\bullet(SU(n)/SO(n), \C) \\
&= H^\bullet(U(1),\C) \otimes H^\bullet(\gl_n, K_\infty, \C)
\end{align*}
and the first factor $H^\bullet(U(1),\C) \cong H^1(\gl_n, O(n), \C)$, which is the cohomology of the circle which is trivial in all degrees except in degrees $0$ and $1$ where it is isomorphic to $\C$, must be excluded to obtain the desired summand $\hbul(\gl_n, K_\infty, \C)$.
\end{proof}

\vspace{0.5cm}
\noindent Before we proceed further let us recall a result of Li and Schewermer \cite[Propositions 5.2 and 5.8]{LS} which are adapted to our needs after changing the notation found therein. First let us recall that $\dim G(\R) = n^2$, $\dim O(n) = n(n-1)/2$, $\rank G(\R) = n$, and $\rank O(n) = (n-1)/2$. Consider a variant symmetric space $X_{\sym}^\prime = G(\R)/O(n)$ of the symmetric space $X_{\sym}$. Then $\dim X_{\sym}^\prime = {n+1 \choose 2}$. let 
\begin{equation}\label{cusp_interval_endpoints}
a(n) := \frac{1}{2}\Bigg( {n+1 \choose 2} -\Big(\frac{n+1}{2}\Big) \Bigg) \; \; , 
b(n) := \frac{1}{2} \Bigg( {n+1 \choose 2} + \Big(\frac{n+1}{2} \Big) \Bigg).
\end{equation}

\newcommand{\irr}{\textrm{irr}}
\noindent Consider the \textit{integer} intervals, where $\dim X_{\sym} = {n+1 \choose 2} - 1$.
$$
I   := [0, \dim X_{\sym}], \; \; 
I_! := (0, a(n)), \; \; 
I_{\cusp} := [a(n),b(n)], \; \; 
I_{\irr} := (b(n), \dim X_{\sym})
$$
where we have the usual notation of intervals, yet, let us explain, for instance $I_{!}$, which is the set of all \textit{integers} $k$ such that $k$ is strictly greater than $0$  and strictly less than $a(n)$.  Note that $I$ is the disjoint union:
$$
I = \Set{0, \dim X_{\sym}} \cup I_{!} \cup I_{\cusp} \cup I_{\irr}.
$$
\begin{remark}
The notation hints at the final theorem of our paper, namely the interval $I_{!}$ hints that the inner cohomology classes that are not cuspidal occur only in this interval, while $I_{\irr}$ hints that the interval in question is `irrelevant'.
\end{remark}
\begin{thm}\label{Li_Schwermer}
(Li and Schwermer) \cite[Propositions 5.2(ii) and 5.8]{LS}
\begin{enumerate}
\item\label{cuspidal_LS}
$H^k_{\cusp}(\lss, \sheafm)=0$ for  $k \in I \setminus I_{\cusp}$.
\item\label{restriction_LS}
The restriction $r^k : H^k(\lss, \sheafm) \to H^k(\partial \lss, \sheafm)$ is an isomorphism for $k > b$, i.e. for all $k \in I_{\irr}$.
\end{enumerate}
\end{thm}

\noindent Consider the following table containing the intervals where cuspidal cohomology may be nontrivial, for primes $n =2,3,5,7,11$, where in $a(n),b(n)$ are defined as in \eqref{cusp_interval_endpoints}, and $S^0$ is the subset of the interval $I= [0, \dim X_{\sym}]$ defined by
$$
S^0:= \Set{ 2l -1 | 1 < l \leq n, \; l \trm{ odd } }
$$
given by the conclusion of the Lemma \ref{Haefliger_lemma}.

\begin{table}\label{cusp_intervals}
\caption{Inner cohomology degrees}
\begin{tabular}{| c | c  | c | c | c | c |} 
\hline
$n$ &  $\dim X_{\sym} = {n+1 \choose 2} -1$ & $I_{\cusp} = [a(n),b(n)]$ &  $S^0$   \\
\hline
$2$ & $2$ & $[3/4,9/4] = \Set{1,2}$  & $\emptyset$\\
\hline
$3$ & $5$ & $[2,4]$ & $\Set{5}$ \\

\hline
 $5$  & $14$ & $[6,9]$ & $\Set{5,9}$ \\

\hline
 $7$ & $27$ & $[12,16]$ & $\Set{5,9,13}$ \\

\hline
$11$ & $65$ & $[30,36]$ & $\Set{5,9,13,17,21}$\\ 
\hline
\end{tabular}
\end{table}

\begin{lem}\label{deg_0_n}
The following isomorphisms hold:
\begin{align*}
H^{\dim X_{\sym}}_{!/\cusp}(\lss, \sheafm) &= H^{\dim X_{\sym}}_{!}(\lss, \sheafm) \\
&\cong H^{\dim X_{\sym}}_{(2)}(\lss, \sheafm) = H^{\dim X_{\sym}} (\lss, \sheafm) \\
& \cong H^0_{!}(\lss, \sheafm) \cong H^0_{c}(\lss, \sheafm) \\
& \cong H^{0}_{!/\cusp}(\lss, \sheafm).
\end{align*}
\end{lem}

\begin{proof}
The boundary cohomology $H^{\dim X_{\sym}}_{\partial}$ in degree $\dim X_{\sym}$ is trivial, since $\dim \partial \bsc = \dim X_{\sym} -1$ so that in degrees strictly greater than $\dim \partial \bsc$ the de Rham cohomology is trivial, the claim follows from de Rham isomorphism \eqref{de_Rham_isomorphism}  Therefore the restriction map $r^{\dim X_{\sym}}$, and hence the composition $r^{\dim X_{\sym}} \circ \Phi^{\dim X_{\sym}}_{BG}$ is the zero map. From the definition of the inner cohomology we have then 
\begin{equation}\label{temp_equation}
H^{\dim X_{\sym}}_{!}= H^{\dim X_{\sym}}_{(2)}= H^{\dim X_{\sym}}.
\end{equation}
On the other hand, by the Poincar\'{e} duality, namely that there exists a nondegenerate pairing 
$
H^0_c \times H^{\dim X_{\sym}} \to \C,
$ \eqref{Poincare_duality}
it follows that $H^0_c \cong H^{\dim X_{\sym}}$.

First consider the case where the prime $n \geq 3$. From Table \ref{cusp_intervals}, we see that the interval $0,\dim X_{\sym} \not \in I_{\cusp}$. By Theorem \ref{Li_Schwermer}\eqref{cuspidal_LS}, $H^0_{\cusp} =H^{\dim X_{\sym}}_{\cusp} = 0$.  On the other hand, the map $i^0$ is injective therefore $H^0_{!/\cusp} = H^0_{!} \cong H^0_c$ (see Remark \ref{degree_0_compact_cohomology}). Now consider the case where the prime $n = 2$. Again, from Table \ref{cusp_intervals}, we see that $2 \in I_{\cusp}$ but $0 \not \in I_{\cusp}$. Consider the Borel-Serre long exact sequence at degree $2$,
\begin{center}
\begin{tikzcd}
& &  H^2(\g,K_\infty, \C) \otimes \Res_{f}(\lam) \arrow{d}{\Phi_{BG}^2} \arrow[dashed]{dr}{j_2} & & & \\
\ldots \arrow{r} & H^2_c  \arrow{r}{i^2}  & H^2 \arrow{r}{r^2}  & H^2_{\partial} \arrow{r}{i^{3}} & H^{3}_c \arrow{r} &\ldots 
\end{tikzcd}
\end{center} 
By Lemma \ref{Haefliger_lemma}, $H^2(\g,K_\infty, \C) = 0$, therefore $H^2_{!} = 0$, and now we run through the remaining argument as in the $n \geq 3$ case already considered above.
\end{proof}

\begin{prop}\label{n_all}
\begin{enumerate}
\item \label{prime_2_3}
For primes $n = 2$ and $n= 3$,  
$
H^\bullet_{!/\cusp}(\lss, \sheafm) = 0.
$
\item \label{prime_5_7}
For prime $n = 5$, resp. $n = 7$, $H^k_{!/\cusp}(\lss, \sheafm) =0 $ for all integers $k \in I_{\cusp} \cup I_{\irr} \cup \Set{0}$ except perhaps for $k = 9$, resp. $k = 13$.

\item \label{prime_11}
For all primes $n \geq 11$, $H^k_{!/\cusp}(\lss, \sheafm) =0 $ for all integers $k \in I_{\cusp} \cup I_{\irr} \cup \Set{0}$.
\end{enumerate}
\end{prop}

\begin{proof}
By Lemma \ref{Wigner_lemma} we may  (and do) restrict our attention to the constant coefficient system. By Lemma \ref{Haefliger_lemma}  $H^k_{!/\cusp} = 0$ for all $k \not \in S^0$.

First we prove assertion \ref{prime_11}. From Table \ref{cusp_intervals} it is clear that $I_{\cusp} \cap S^0 = \emptyset$ and in fact these are precisely such primes. Indeed, the maximal element in $S^0$ is $2n-1$ corresponding to $l = n$, and the minimal value of $I_{\cusp}$ is $a$ (see \eqref{cusp_interval_endpoints}). Solving for $n$ for which the inequality $2n-1 \leq a $ holds,  we obtain the desired claim. Now, it follows from Theorem \ref{Li_Schwermer}\eqref{cuspidal_LS}, and the definition of the inner cohomology that $H^k_{!/\cusp} = 0$ for all $k \in I_{\cusp} \cup I_{\irr}$. As for the vanishing of $H^0_{!/\cusp}$, the restriction $r^{\dim X_{\sym}}$ is the zero map. Indeed, by Theorem \ref{Li_Schwermer}\eqref{restriction_LS}  for all $k > b$ the the restriction map $r^k: H^k(\bsc, \C) \to H^k(\partial \bsc, \C)$ is an isomorphism. On the other hand, since $\dim \partial \bsc = \dim \bsc - 1 = \dim X_{\sym} - 1$ it follows that $H^{\dim X_{\sym}}_\partial$, \textit{hence} $H^{\dim X_{\sym}}$, is trivial. But then by Lemma \ref{deg_0_n}, $H^{0}_{!/cusp}$ is also trivial.

Now consider assertion \ref{prime_2_3}. For prime $n =2$, the set $S^0$ is empty, and so by Lemma \ref{Haefliger_lemma} the inner cohomology vanishes for all degrees $k = 1, 2$, and hence by Lemma \ref{deg_0_n} in degree $k = 0$ too. The argument  assertions \ref{prime_2_3} for prime $n=3$ and \ref{prime_5_7} is same as that of the proof of the assertion \ref{prime_11} above, provided the exceptions mentioned in the assertion \ref{prime_5_7} are excluded in the argument.
\end{proof}

\begin{remark}\label{Notation_Haefliger}

\begin{enumerate}
\item
Proposition \ref{n_all} implies that we may restrict our attention to the interval $I_{!}$ to find inner cohomology classes that are not cuspidal, justifying the notation $!$ in $I_{!}$.
\item

The subset $S^0$ may be viewed as the subset of those integers in $I$ for which the inner cohomology \textit{may} be nontrivial, so that for integers in the complement in $I$ of this set $S^0$ the inner cohomology is trivial. This also explains the choice of notation $S^0$ with $0$ in the superscript position.
\end{enumerate}
\end{remark}

\noindent Now we prove the main result of the article. Let  
\begin{equation*}\label{analysis_2_residual_part_of_Borel_Garland_map}
\Res_{f}(\lam) :=  \bigoplus_{ \substack{\pi_f \in \coh_{(2)}(G, K_f, \lam) \\
\textrm{type}(\pi_f) = \omega_\lam^{1/n}} }
\pi_f.
\end{equation*}
The following Borel-Serre fundamental long exact sequence \eqref{Borel_Serre_fundamental_exact_sequence} at degree $d$ is a useful illustration of the following theorem.

\begin{center}
\title{Approximation by Borel Garland map}
\begin{tikzcd}
& &  H^d(\g,K_\infty, \C) \otimes \Res_{f}(\lam) \arrow{d}{\Phi_{BG}^d} \arrow[dashed]{dr}{j_d} & & & \\
\ldots \arrow{r} & H^d_c  \arrow{r}{i^d}  & H^d \arrow{r}{r^d}  & H^d_{\partial} \arrow{r}{i^{d+1}} & H^{d+1}_c \arrow{r} &\ldots 
\end{tikzcd}
\end{center}

\begin{thm}\label{main_thm}
Assume that $n$ is a prime number. For all primes $n \geq 2$, the quotient module $H^\bullet_{!/\cusp}(\lss, \sheafm)$ vanishes if $\sheafm$ is not isomorphic to the constant sheaf $\C$. So, suppose otherwise, i.e. $\sheafm \cong \C$, and let 
$S^0 = \Set{ 2l-1 | 1 < l \leq n, \; l \trm{ odd }}$, then  

\begin{enumerate}
\item 
for prime $n =2,3$, the module $H^\bullet_{!/\cusp}(\lss, \C) = 0$, and

\item 
for all primes $n \geq 5$,
\begin{equation}\label{possible_cases}
H^k_{!/\cusp}( \lss, \C) \cong \begin{cases}
0  & \trm{ for } k \not \in S^0.  \\
\ker( r^k|_{\Phi^k_{BG}(\Res_f(\lam)}) & \trm{ for } k \in S^0.
\end{cases}
\end{equation}
\end{enumerate}
\end{thm}

\begin{proof}
The first assertion and the first part of the second assertion is  Proposition \ref{n_all}. The second part of the second assertion is a consequence of Proposition \ref{residual_decomposition} and the definition of the inner cohomology.
\end{proof}

\begin{remark}
As a final remark, we note that the some of the conclusions of Theorem \ref{main_thm} \textit{could} be  more precise trading for simplicity. For instance, in the case $n = 5$, it follows from the facts $I_{\cusp} = [ 6, 9]$, $S^0 = \Set{5,9}$ one has $H^5_{!} = H^5_{!\cusp}$ because $H^5_{\cusp} = 0$, and similarly for the other cases too.
\end{remark}

\section*{References}
\bibliographystyle{plainurl}
\bibliography{document}

\begin{thebibliography}{10}

\bibitem{Ar1}
James Arthur.
\newblock Eisenstein series and the trace formula.
\newblock In {\em Automorphic forms, representations and L-functions (Proc.
  Sympos. Pure Math., Oregon State Univ., Corvallis, Ore., 1977), Part},
  volume~1, pages 253--274, 1979.

\bibitem{Ar}
James Arthur.
\newblock {$L^2$}-cohomology and automorphic representations.
\newblock {\em Canadian Mathematical Society}, 3:1--17, 1995.

\bibitem{Bo1}
Armand Borel.
\newblock Regularization theorems in {L}ie algebra cohomology. applications.
\newblock {\em Duke Mathematical Journal}, 50(3):605--623, 1983.

\bibitem{Bo}
Armand Borel.
\newblock {\em Linear algebraic groups}, volume 126.
\newblock Springer Science \& Business Media, 2012.

\bibitem{BC}
Armand Borel and William Casselman.
\newblock {$L^{2}$}-cohomology of locally symmetric manifolds of finite volume.
\newblock {\em Duke Mathematical Journal}, 50(3):625--647, 1983.

\bibitem{BG}
Armand Borel and Howard Garland.
\newblock Laplacian and the discrete spectrum of an arithmetic group.
\newblock {\em American Journal of Mathematics}, 105(2):309--335, 1983.

\bibitem{BS}
Armand Borel and Jean-Pierre Serre.
\newblock Corners and arithmetic groups.
\newblock {\em Commentarii Mathematici Helvetici}, 48(1):436--491, 1973.

\bibitem{BW}
Armand Borel and Nolan~R Wallach.
\newblock {\em Continuous cohomology, discrete subgroups, and representations
  of reductive groups}, volume~67.
\newblock American Mathematical Soc., 2013.

\bibitem{Hae}
Andr{\'e} Haefliger.
\newblock Differential cohomology.
\newblock In {\em Differential topology}, pages 19--70. Springer, 2010.

\bibitem{Ha}
G.~Harder and K.~Diederich.
\newblock {\em Lectures on Algebraic Geometry I: Sheaves, Cohomology of
  Sheaves, and Applications to Riemann Surfaces}.
\newblock Aspects of Mathematics. Springer Fachmedien Wiesbaden, 2011.

\bibitem{HR}
G~Harder and A~Raghuram.
\newblock {E}isenstein cohomology for ${GL(N)}$ and ratios of critical values
  of rankin-selberg ${L}$--functions, 2014.
\newblock URL: \url{https://arxiv.org/abs/1405.6513}.

\bibitem{Ha2}
G{\"u}nter Harder.
\newblock {E}isenstein cohomology of arithmetic groups. the case {$GL_2$}.
\newblock {\em Inventiones mathematicae}, 89(1):37--118, 1987.

\bibitem{Ha1}
G{\"u}nter Harder.
\newblock Cohomology of arithmetic groups.
\newblock {\em Preprint}, 2006.
\newblock URL:
  \url{http://www.math.uni-bonn.de/people/harder/Manuscripts/buch/}.

\bibitem{JS}
Herv{\'e} Jacquet and Joseph~A Shalika.
\newblock On {E}uler products and the classification of automorphic
  representations i.
\newblock {\em American Journal of Mathematics}, 103(3):499--558, 1981.

\bibitem{La}
Robert~P Langlands.
\newblock {\em On the functional equations satisfied by {E}isenstein series},
  volume 544.
\newblock Springer, 2006.

\bibitem{LS}
Jian-Shu Li and Joachim Schwermer.
\newblock On the {E}isenstein cohomology of arithmetic groups.
\newblock {\em Duke Mathematical Journal}, 123(1):141--169, 2004.

\bibitem{MW}
Colette M{\oe}glin and J-L Waldspurger.
\newblock Le spectre r{\'e}siduel de {$GL_n $}.
\newblock In {\em Annales Scientifiques de l'{\'E}cole Normale Sup{\'e}rieure},
  volume~22, pages 605--674. Elsevier, 1989.

\end{thebibliography}

\end{document}